\documentclass[a4,10pt,reqno,fleqn]{amsart}

\title[Automorphisms of projective tensor product]{Automorphisms of Banach space projective tensor product of $C^*$-algebras}

\author[R. Jain]{Ranjana Jain}
\address{Department of
  Mathematics\\ University of Delhi\\ Delhi-110007, INDIA.}
\email{rjain@maths.du.ac.in}

\usepackage[all]{xy}
\usepackage{amsmath,amsfonts,amssymb,amsthm,a4,hyperref}

\usepackage[margin=4cm]{geometry}
\usepackage{hyperref}
\usepackage{cleveref}
\usepackage{setspace}
\usepackage{enumitem}
\usepackage{verbatim}

\newtheorem{theorem}{\sc Theorem}[section]
\newtheorem{cor}[theorem]{\sc Corollary}

\newtheorem{lemma}[theorem]{\sc Lemma}
\numberwithin{equation}{section}

\theoremstyle{remark}

\newtheorem{rem}[theorem]{\sc Remark}

\newcommand{\oop}{\widehat\otimes}
\newcommand{\seq}{\subseteq}
\newcommand{\oh}{\otimes^h}
\newcommand{\omin}{\otimes^{\min}}
\newcommand{\obp}{\otimes^\gamma}
\newcommand{\oi}{\otimes^\lambda}
\newcommand{\ra}{\rightarrow}
\newcommand{\ot}{\otimes}

\newcommand{\Z}{\mathcal{Z}}

\newcommand{\C}{\mathbb{C}}
\newcommand{\N}{\mathbb{N}}
\newcommand{\M}{\mathbb{M}}

\begin{document}


\begin{abstract}
 For unital $C^*$-algebras $A$ and $B$, we completely characterize the isometric ($*$-) automorphisms of their  Banach space projective tensor product $A\obp B$. This leads to the characterization of inner and outer isometric $*$-automorphisms of $A\obp B$, as well. As an application, we provide a partial affirmative answer to a question posed by Kaijser and Sinclair, viz., we prove that for unital $C^*$-algebras $A$ and $B$, the set of norm-one unitaries of $A\obp B$ coincides with $U(A) \otimes U(B)$, where $U(A)$ is the unitary group of $A$. We also establish the fact that the relative commutant of $A\obp \C 1$ in $A \obp B$ is same as $\Z(A) \obp B$, where $B$ is a subhomogenous unital $C^*$-algebra, and $A$ is any $C^*$-algebra. 
\end{abstract}

\maketitle

\section{Introduction}

For Banach spaces $X$ and $Y$, the projective tensor norm on $X \ot Y$ is defined as
$$ \|u\| = \inf \Big\{ \sum_{i=1}^n \|x_i\|  \|y_i\|: u = \sum_{i=1}^n x_i \ot y_i \Big\}, $$
for all $u\in X \ot Y$. The completion of $X\ot Y$ with respect to this norm is defined as the {\it projective tensor product} of $X $ and $Y$, and is denoted by $X\obp Y$.
It is known that if $X$ and $Y$ are Banach $*$-algebras, then so is $X\obp Y$.

In 1970, C. C. Graham (\cite{graham1}, \cite[Theorem $11.7.1$]{graham}) characterized the automorphisms of $A\obp B$ in terms of the automorphisms and isomorphisms between $A$ and $B$. In particular, he proved that for unital abelian $C^*$-algebras $A$ and $B$ with no non-trivial projections, an algebra automorphism $\theta$ of $A \obp B$ is either of the form $\theta (a \otimes b)= \phi(a)\otimes \psi(b)$, where $\phi:A\ra A$ and $\psi : B\ra B$ are automorphisms, or, of the form $\theta (a \otimes b)= \mu(b) \otimes \rho(a)$, where $\mu : B \ra A$ and $\rho : A \ra B$ are isomorphisms. The isometric $*$-automorphisms of $A \oop B$ and isometric automorphisms of $A \oh B$ have also been characterized in \cite{rjak}, where $A$ and $B$ are unital $C^*$-algebras; $\oop$ and $\oh$ are operator space projective tensor product and Haagerup tensor product respectively. In this article we prove its analogue for Banach space projective tensor product of unital $C^*$-algebras (not necessarily abelian).

Kallman \cite[Corollary 1.14]{kall}, in 1969, established that if $R$ and $S$ are von Neumann algebras and $\phi$ and $\psi$ are $*$-automorphisms of $R$ and $S$ respectively, then $\phi \otimes \psi$ is outer on $R \bar{\otimes} S$ if and only if either $\phi$ or $\psi$ is outer. Further, in 1975, S. Wasserman \cite{wass} proved that for unital $C^*$-algebras $A$ and $B$ with $*$-automorphisms $\alpha$ and $\beta$, $\alpha \omin \beta$ is inner if and only if $\alpha$ and $\beta$ are both inner. He also discussed the same question for some other $C^*$-norms.  Also, Bunce \cite[Theorem 1]{bunce} proved that for unital Banach algebra $A$, the flip map $\tau: A \obp A \to A \obp A$ defined as $\tau(a\ot b) = b \ot a$ is an inner automorphism if and only if $A$ is a matrix algebra. Using few of their techniques and the above characterization of automorphisms, we provide a complete characterization of inner and outer automorphisms of $A\obp B$ for unital $C^*$-algebras. 

In 1984, Kaijser and Sinclair \cite{kai-sinc}, proved that for unital Banach algebras $A$ and $B$ with one of them having approximation property, $U_0(A \obp B) = U_0(A) \ot U_0(B):= \{a\ot b: a\in U_0(A), b\in U_0(B)\}$, where $U_0(A)$ is the subgroup of the unitary group $U(A):= \{u\in A: u^{-1} \in A, \|u\|=\|u^{-1}\| =1 \}$ generated by $\{\exp ih:h\in A, h \,\text{Hermitian}\}$.  They further asked under what conditions on unital Banach algebras $A$ and $B$, $ U(A \obp B) = U(A) \ot U(B) $. We prove an appropriate version of this equality for Banach space projective tensor products of unital $C^*$-algebras. Note that this result is not true for the spatial tensor product of $C^*$-algebras, that is, for unital $C^*$-algebras $A$ and $B$, $U(A\omin B)$ need not be same as $U(A) \ot U(B)$ as was illustrated in \cite[Remark 2.6]{vr-1}. 


Further, in 1973, Haydon and Wasserman \cite{hay-wass} proved that for $C^*$-algebras $A$ and $B$, with $B$ unital,  the relative commutant of $A \omin \C 1$ in $A\omin B$ coincides with $Z(A) \omin B $. This result was later extended by Archbold \cite{arch} to any $C^*$-norm. Also R. R. Smith, in 1991 (see \cite[Corollary 4.7]{smith}), established a strong version of Tomita's Commutant Theorem for Haagerup tensor product. In particular, he proved that for unital subalgebras $A_1$ and $B_1$ of $C^*$-algebras $A$ and $B$ respectively, the relative commutant of $A_1 \oh B_1$ in $A\oh B$ is same as $ A'_1 \oh B'_1$; $A'_1, B'_1$ being the relative commutants of $A_1$, $B_1$ in $A$ and $B$ respectively. We prove an analogue for projective tensor product in a specific case. 

Let us give a brief outline of the present article. In Section 2, we prove that for unital $C^*$-algebras $A$ and $B$, an isometric (resp., isometric $*$-) automorphism $\theta$ of $A\obp B$ is precisely of the form $\phi \obp \psi$, or $(\mu \obp \rho) \circ \tau$, where $\phi:A \to A, \psi:B \to B$, $\mu: B \ra A, \rho : A\ra B$ are isometric (resp., $*$-) isomorphisms, and $\tau: A\obp B\to B\obp A$ is the flip map given by $\tau (a\ot b ) = b\ot a$. Using this characterization we further deduce that, if one of $A$ or $B$ is not a matrix algebra, then the inner $*$-automorphisms of $A\obp B$ are precisely of the form   $\phi \obp \psi$, for some inner $*$-automorphisms $\phi$ and $\psi$ of $A$ and $B$ respectively. This enables us to describe the precise form of inner and outer $*$-automorphisms of $B(H) \obp B(H)$. Using this form of inner $*$-automorphisms of $B(H) \obp B(H)$, we further deduce that  $ U_1(A \obp B) = U(A) \ot U(B)$, for unital $C^*$-algebras $A$ and $B$, where $ U_1(A \obp B)$ represents the set of norm-one unitaries of $A\obp B$. In the last section, we establish that the relative commutant of $A\obp \C 1$ (resp., $A \oop \C 1$) in $A\obp B$ (resp., $A\oop B$) is isometrically isomorphic (resp., $*$-isomorphic) to $\Z(A) \obp B $ (resp., $\Z(A) \oop B $),  $A$ being a $C^*$-algebra and $B$ any subhomogenous unital $C^*$-algebra.

\section{Isometric Automorphisms of $A \obp B$}

 Let $A$ be a unital Banach algebra. For $a\in A$, the {\it numerical range} $V(a)$ of $a$ is defined by $$ V(a):= \{ f(a): f \in A^*, \|f\| =f(1) =1\}. $$
An element $h\in A$ is said to be {\it Hermitian} if $V(h) \seq \mathbb{R}$. We write $H(A)$ for the Banach space (over reals) of Hermitian elements of $A$. In case of $C^*$-algebras, Hermitian elements coincide with the self adjoint elements. Hermitian elements of the projective tensor product of certain Banach algebras were identified by Kaiser and Sinclair as follows:
 
 \begin{theorem}\cite[Theorem 3.1]{kai-sinc}
Let $A$ and $B$ be unital Banach algebras such that the canonical map $i:A\obp B \to A \oi B$ is injective, where $\oi$ is the injective tensor product of Banach spaces.  Then, $$H(A \obp B) = H(A) \ot 1 + 1 \ot H(B):=\{a\ot 1 + 1\ot b;\,\,a\in H(A), b\in H(B)\}.$$
 \end{theorem}
 
 \begin{rem}\label{herm}
  If $A$ and $B$ are $C^*$-algebras, then Haagerup \cite{haag} proved that the canonical map $i:A\obp B \to A \oi B$ is injective. Thus, for $C^*$-algebras $A$ and $B$, the Hermitian elements of $A\obp B$ are determined by the above theorem.
 \end{rem}

To prove the main theorem we need the following result whose proof is quite elementary.  For a Banach $*$-algebra $A$, let $Aut(A)$ denote the set of $*$-automorphisms of $A$. Note that a $*$-isomorphism between $C^*$-algebras is an isometry. 
 
 \begin{lemma}\label{iso-isom}
  Let $A_i$ and $B_i$ be Banach algebras and $\phi_i:A_i\ra B_i$,  $i=1,2$, be isometric isomorphisms. Then, the canonical mapping $\phi_1 \ot \phi_2: A_1\ot A_2 \ra B_1\ot B_2$ extends uniquely to an isometric isomorphism $\phi_1 \obp \phi_2: A_1\obp A_2 \ra B_1\obp B_2$ such that $(\phi_1 \obp \phi_2)(\sum_{i=1}^n (x_i \ot y_i)) = \sum_{i=1}^n\phi_1(x_i) \ot \phi_2(y_i)$. 
 \end{lemma}

 \begin{theorem}\label{iso-auto}
  Let $A$ and $B$ be unital $C^*$-algebras. Then a mapping $\theta: A\obp B \to A \obp B$ is an isometric (resp., isometric $*$-) automorphism if and only if either $\theta = \phi \obp \psi$, or $\theta = (\mu \obp \rho) \circ \tau$, where $\phi:A \to A, \psi:B \to B$, $\mu: B \ra A, \rho : A\ra B$ are some isometric (resp., $*$-) isomorphisms, and $\tau: A\obp B\to B\obp A$ is the flip map given by $\tau (x\ot y ) = y\ot x$. 
 \end{theorem}
 
 \begin{proof}
  Let $\theta$ be an isometric automorphism of $A \obp  B$. We first claim that $\theta$ maps $H(A)\otimes 1$ into $A \otimes 1$ or $1\otimes B$. Since $\theta $ preserves identity and is an isometry, $V(\theta (x)) \subseteq V(x)$ for all $x\in A\obp B$, which further implies that $\theta(H(A\obp B)) \seq H(A\obp B)$. 
Thus, Remark \ref{herm} gurantees that $\theta$ leaves $H(A)\otimes1 + 1\otimes H(B)$ invariant. For $x \in H(A)$,
$$ \theta(x\otimes 1) = u\otimes 1 + 1\otimes v, $$
for some $u\in H(A)$ and $v\in H(B)$. Using the relation
$$\theta (x^2\otimes 1) = u^2\otimes1 + 2u\otimes v +1\otimes v^2,$$
we get $u\otimes v \in H(A\obp B)$. If $u \notin \C1_A$ and $v \notin \C1_B$, then by Hahn-Banach Theorem, we can choose $f\in A^*$ and $g\in B^*$ such that
$f(1_A)= 0 =g(1_B)$, and $f(u) \neq 0,\, g(v) \neq 0$. So $(f \otimes g)(u
\otimes v) \neq 0$ and $(f \otimes g)(A \otimes1 + 1 \otimes B)= (0)$,
which contradicts the fact that $u \otimes v \in A \otimes 1 +
1\otimes B$. Therefore, either $u = \alpha 1_A$ or $v = \beta 1_B$, for
some $\alpha,\, \beta \in \mathbb{C}$, giving that  $\theta (x \otimes 1)$ is
either in $A \otimes 1$ or in $1 \otimes B$, and this is true for all $x \in H(A)$. Let, if possible, there exist $x,y \in H(A)$ such that
$$
\theta (x\otimes 1)=a\otimes 1, \ \theta(y\otimes 1)=1\otimes b,
$$
and neither $a$ nor $b$ is a multiple of $1$. Then, for $x+y\in H(A)$,
$$ \theta ((x+y)\otimes 1)= a\otimes 1+1\otimes b,$$
which is neither in $A\otimes 1$ nor in $1\otimes B$, a
contradiction. Thus, $\theta (H(A)\otimes 1) \seq A \otimes 1$ or $\theta (H(A)\otimes 1) \seq 1\otimes B$.
Since $A$ is the complex linear span of $H(A)$, it follows that $\theta(A\otimes 1) \seq A\otimes 1$ or $\theta(A\otimes 1) \seq1\otimes B$. Repeating the process for the isometric automorphism $\theta^{-1}$ on $A \obp B$, , it is easy to
see that  $\theta(A\otimes 1) = A\otimes 1$ or $ 1\otimes B$. In the former case, define a map $\phi:A\ra A$ as $\phi(a) = a'$ where $\theta(a\ot 1) = a'\ot 1$. So that $\theta(x\otimes1)=\phi(x)\otimes 1$, for all $x\in A$. Also, in the latter case, there exists a map $\rho:A\ra B$ such that $\theta(x\otimes1)= 1\otimes \rho(x)$ for all $x\in A$.
Similarly, $\theta$ maps $1\otimes B$ onto $A\otimes 1$ or onto $1\otimes B$.  Thus, either $\theta(1\otimes y)= 1\otimes \psi(y)$ for some $\psi:B\ra B$, or $\theta(1\otimes y)=\mu(y)\otimes 1$ for some $\mu : B \ra A$. It is easy to check that $\phi,\, \rho,\, \psi$ and $\mu$ are all isometric isomorphisms ($*$-maps if $\theta $ is $*$-map). 
Now, either $\theta(x\otimes y)= \phi(x)\otimes \psi(y)$ or $\theta(x\otimes y)= \mu(y)\otimes \rho(x)= ((\mu \ot \rho)\circ \tau) (x\ot y)$, for all $x \in A,\, y \in B$, as the other two cases will lead us to the fact that $\theta(A \obp B)$ is either contained in $A\obp 1$ or in $1 \obp B$.

Converse follows from Lemma \ref{iso-isom}, and the fact that $\tau$ is an isometric $*$- isomorphism. 
 \end{proof}

 This result now leads us to characterize the isometric inner $*$-automorphisms of $A\obp B$.
 \begin{theorem}\label{inner}
  Let $A$ and $B$ be unital $C^*$-algebras such that at least one of them is different from matrix algebra. Then a mapping $\theta: A\obp B \to A \obp B$ is an isometric inner $*$-automorphism if and only if $\theta = \phi \obp \psi$, where $\phi \in Aut(A)$ and $\psi \in Aut(B)$ are inner. 
 \end{theorem}
 
 \begin{proof}
  Let $\theta$ be an isometric inner $*$-automorphism. By Theorem \ref{iso-auto}, either $\theta = \phi \obp \psi$, or $\theta = (\mu \obp \rho) \circ \tau$, where $\phi:A \to A, \psi:B \to B$, $\mu: B \ra A, \rho : A\ra B$ are $*$-isomorphisms, and $\tau: A\obp B\to B\obp A$ is the flip map.
  
  Let us assume that $\theta = (\mu \obp \rho) \circ \tau$. In this case, we assert that $A$ and $B$ are both matrix algebras. We first claim that $A$ is a simple algebra. Let $I$ be a proper closed ideal of $A$. By \cite[Lemma 3.12]{GJ2}, there is a quotient map $\pi \obp Id: A\obp B \to A/I \obp B$ with $\ker (\pi \obp Id) = I \obp B$, where $\pi:A \to A/I$ is the canonical quotient map. Since $\theta $ is an inner automorphism, $\theta( I\obp B) \seq I \obp B$. Now for any $a\in I$, $\theta(a\ot e_B) \in \ker (\pi \obp Id)$, $e_B$ being the identity of $B$. This gives $(\pi \obp Id) ( \rho(e_B) \ot \mu(a))= 0$, that is, $(e_A+ I)\ot \mu(a) = 0$. This further shows that $a=0$, using the fact that $I$ is a proper and $\mu$ is one-one. Thus $A$, and similarly $B$ is simple. It is now sufficient to show that $A$ is finite dimensional. Since $\theta$ is an inner $*$-automorphism, there exists a unitary $ u \in A\obp B$ such that $\theta (x) = uxu^*$, for all $x\in A\obp B$. So, for any $a\in A, b\in B$, we have $$ \theta (a\ot b) = \rho(b) \ot \mu(a) = u(a\ot b) u^*. $$  
  For above $u\in A\obp B$, choose $z, w\in A\ot B$ satisfying 
  $$ \|u-z\| < \frac{1}{4\|u\|}, \quad \|z\| <\|u\| +1 \quad \text{and}\quad \|u^* - w\| < \frac{1}{4(\|u\|+1)}.$$
  Then, for any $a\in A$ and $b\in B$, we obtain
  \begin{eqnarray*}
   \|\rho(b) \ot \mu(a) - z(a\ot b) w\|  & \leq &  \|\rho(b) \ot \mu(a) - u(a\ot b)u^*\| +  \| u(a\ot b)u^* - z(a\ot b)u^*\| \\
       & & +\| z(a\ot b)u^* - z(a\ot b)w\| \\
   & \leq & \|u-z\|\|a\ot b\|\|u^*\| +  \|u^*-w\|\|a\ot b\|\|z\| \\
   & \leq & \frac{1}{2} \|a\| \|b\|.   
  \end{eqnarray*}
  For any $b\in B$, we have
$$ \|\rho(b) \ot \mu(e_A) - z(e_A\ot b) w\|   \leq  \frac{1}{2} \|b\|. $$
If $z =\sum_{i=1}^r x_i \ot y_i$ and  $z =\sum_{j=1}^s u_j \ot v_j$, then 
$$ \|\rho(b) \ot e_B - \sum_{i,j} x_i u_j \ot y_ibv_j \|   \leq  \frac{1}{2} \|b\|. $$
 
Now, choose $f\in B^*$ with $f(e_B) = 1=\|f\|$, and consider the (bounded) left slice map $L_f: A\obp B \to A$ defined as $L_f(\sum_{i=1}^na_i \ot b_i) = \sum_{i=1}^nf(b_i)a_i$. One can easily see that $L_f$ is a contraction, so the above inequality yields
$$ \| \rho(b) - \sum_{i,j} f(y_ibv_j)x_iu_j\| \leq \frac{\|b\|}{2},$$
and this relation is true for all $b\in B$. Since $\rho:B \to A$ is an isometric isomorphism, for any $a\in A$, we obtain
$$ \| a - \sum_{i,j} f(y_i \rho^{-1}(a)v_j)x_iu_j\| \leq \frac{\|a\|}{2}.$$
If we denote $\hat{A}$ by the closed linear span of $\{x_iu_j: 1\leq i \leq r, 1 \leq j \leq s\}$, then by Riesz Lemma, $A = \hat{A}$. Hence $A$ and similarly $B$ is finite dimensional, and both are matrix algebras. 

Thus $\theta = \phi \obp \psi$, and it remains to prove that $\phi$ and $\psi$ are inner. Consider the canonical identity map $i: A \obp B \to A\omin B$, which is one-one, by \cite{haag}. For the quotient map $\phi \omin \psi: A\omin B \to A\omin B$, it is  easy to verify that $ i\circ (\phi \obp \psi) = (\phi \omin \psi) \circ i$. So, for any $x\in A\obp B$, we have $$ (\phi \omin \psi)(i(x)) = i(u) i(x) (i(u))^*.$$
Since $i(A\obp B) $ is dense in $A\omin B$, $\phi \omin \psi$ is inner. By \cite[Theorem 1]{wass}, $\phi$ and $\psi$ are both inner. 

Converse is direct.
\end{proof}

Following the last part of the above proof, we  obtain a characterization for the outer $*$-automorphisms of $A\obp B$.

\begin{cor}
Let $\phi$ and $\psi$ be $*$-automorphisms of unital $C^*$-algebras $A$ and $B$ respectively. Then $\phi \obp \psi$ is an outer $*$-automorphism of $A\obp B$ if and only if either $\phi$ or $\psi$ is an outer $*$-automorphism.
\end{cor}

\begin{rem}
 It can be noticed from the above proof that an automorphism $(\mu \obp \rho) \circ \tau$, where $\mu: B \ra A$ and $ \rho : A\ra B$ are isomorphisms (not necessarily $*$-preserving), can not be inner.
\end{rem}

\begin{cor}
 Let $A$ and $B$ be unital $C^*$-algebras with one of them different from a matrix algebra. Then, a mapping $\theta: A\obp B \to A \obp B$ is an isometric outer $*$-automorphism if and only if either $\theta = \phi \obp \psi$, where $\phi \in Aut(A)$, $\psi \in Aut(B)$ with at least one being an outer $*$-automorphism, or, $\theta = (\mu \obp \rho) \circ \tau$, where $\phi:A \to A, \psi:B \to B$, $\mu: B \ra A, \rho : A\ra B$ are $*$-isomorphisms, and $\tau: A\obp B\to B\obp A$ is the flip map.
\end{cor}

For a Hilbert space $H$, every $*$-automorphism of $B(H)$ is an inner automorphism (see, \cite{take2}). This yields the following characterization:

\begin{cor}\label{inn-B(H)}
The isometric inner $*$-automorphisms of $B(H) \obp B(H)$ are precisely of the form $\phi \obp \psi$, where $\phi,\psi \in Aut(B(H))$; $H$  being an infinite dimensional separable Hilbert space. 
\end{cor}

For a separable infinite dimensional Hilbert space $H$, $B(H)$ has no outer $*$-automorphism. However this is not the case with $B(H) \obp B(H)$, which has plenty of outer $*$-automorphisms.  
\begin{cor}\label{out-B(H)}
  The isometric outer $*$-automorphisms of  $B(H) \obp B(H)$ are precisely of the form $(\phi \obp \psi) \circ \tau$, where $\phi,\psi \in Aut(B(H))$ and $\tau$ is the flip map; $H$ being an infinite dimensional separable Hilbert space.
\end{cor}

\begin{rem}
 Note that if either $A$ or $B$ has an outer automorphism say $\phi$, then $A\obp B$ will also have an outer automorphism namely, $\phi \obp Id$ or $Id \obp \phi$. However, the converse is not true, as can be seen from  Corollary \ref{out-B(H)}. 
\end{rem}

\subsection{Unitary group of $A \obp B$}
  For a unital Banach algebra $A$, its {\em unitary group} is defined as $U(A)=\{u\in GL(A): \|u\| = \|u^{-1}\|=1 \}$. Note that $U(A)$ is not necessarily a subgroup of $GL(A)$. 
  
  If $A$ is a unital $C^*$-algebra, then it is easily seen that its unitary group $U(A)$ coincides with its usual set of unitaries  $\{u \in A: uu^*=u^*u=1\}$. The same need not hold in an arbitrary unital Banach $*$-algebra. However, if $A$ and $B$ are unital $C^*$-algebras, then for any $x=u\ot v \in U(A) \ot U(B)$, we observe that $x^*x=xx^*=1= \|x\|$. In view of this, for a unital Banach $*$-algebra $A$, it is quite appropriate to consider the set of norm-one unitaries $U_1(A\obp B) = \{u \in A: u^*u = uu^*=1= \|u\|\}$. 

The characterization of inner $*$-automorphisms of $B(H) \obp B(H)$ now allows us to determine the norm-one unitary group of $A\obp B$ in terms of unitaries of $A$ and $B$.
 \begin{theorem}
  For unital $C^*$-algebras $A$ and $B$,  
 $$ U_1(A \obp B) = U(A) \ot U(B).$$
 \end{theorem}
 \begin{proof}
  Since $A$ and $B$  can be embedded $*$-isometrically and unitally into $B(H)$ for some  Hilbert space $H$, by \cite[Theorem 3.1]{GJ2}, $A\obp B$ embeds isometrically into  $B(H)\obp B(H)$. For $u\in U_1(A\obp B)$, define $\theta (x)=uxu^*$ for all
$x\in B(H)\obp B(H)$, then $\theta $ is an isometric inner $*$-automorphism of $B(H)\obp B(H)$. So by Corollary \ref{inn-B(H)}, $\theta = \phi \obp \psi$, where $\phi$ and $\psi$ are inner $*$-automorphisms of $B(H)$.    So there exist unitaries $a$ and $b$ in $B(H)$ such that $\phi(x)= axa^*$ and $\psi(x) = bxb^*$ for all $x\in B(H)$. This gives
$$ u(v\otimes w)u^* =\theta(v\ot w)= ava^*\otimes bwb^* = (a\otimes b)(v\otimes w)(a\otimes b)^*,$$
for all $v\otimes w\in B(H)\otimes B(H)$. Thus, $(a\otimes b)^*u$ is an element of $\Z(B(H)\obp B(H))$, which by \cite[Theorem 5.1]{GJ2}, coincides with $\Z(B(H)) \obp \Z(B(H))$. So  $(a\otimes b)^*u=\lambda 1\otimes 1$, where $\lambda\in \mathbb{C}$ with $\left|\lambda\right|=1$, so $u=(\lambda a)\otimes b$. 

Now, for any $w \in A^*$ for which $w(a) \neq 0$, consider the right slice map $R_w: A\obp B \to B$ defined as $R_w(x\ot y) = w(x)y$. Then, $R_w(u) \in B$ which shows that $w(\lambda a) b \in B$, and thus $b$ is an element of $B$. Similarly by taking the left slice map one can see that  $a\in A$ and hence we are done.
 \end{proof}

 \section{Relative commutant}
 For a Banach algebra $A$ and any subset $S$ of $A$, the {\it relative commutant} of $S$ in $A$ is defined as
 $$ S' = \{ a \in A: as = sa,\, \forall s \in S\}.$$
Also, recall that for $C^*$-algebras $A$ and $B$, the {\it Haagerup norm} of an element $u \in A\ot B $ is given by
$$\|u\|_h = \inf \big\{\|\Sigma_i\, a_ia^*_i\|^{1/2}\ \|\Sigma_i\,
b^*_ib_i\|^{1/2}:\, u=\Sigma^{n}_{i=1}a_i\otimes b_i \big\}.$$
The {\it Haagerup tensor product} of $A$ and $B$, denoted as $A \oh B$, is defined as the completion of $A \ot B$ with respect to $\oh$-norm. It is known that $A \oh B$ is a Banach algebra in which involution is not an isometry (except in the trivial cases). Also, for $C^*$-algebras $A$ and $B$, the canonical identity map $i: A\obp B \ra A\oh B$ is an injective homomorphism (see, \cite[Proposition 3.2]{GJ2}). 
\begin{theorem}
 For $C^*$-algebras $A$ and $B$, where $B$ is unital and subhomogenous
 $$(A\obp \C 1)' = \Z(A) \obp B. $$
\end{theorem}

\begin{proof}
 
 Since $ \Z(A) \otimes B \seq (A \ot \C 1)' = (A \obp \C 1)'$, consider the inclusion
function from $ \Z(A) \ot B$ into $ (A \obp \C 1)'$. By \cite[Theorem 3.1]{GJ2}, $\Z(A) \obp B$ can be considered as a
$*$-subalgebra of $A \obp B$ and thus of $(A\obp \C 1)'$, so that for any $u\in \Z(A) \ot B$, $
\|u\|_{\Z(A) \obp B} = \|u\|_{(A\obp \C 1)'} $. Thus, the inclusion
function extends uniquely to an isometric $*$-homomorphism, say,
$\theta$ from $\Z(A) \obp B$ into $(A\obp \C 1)'$. It is sufficient
to establish the surjectivity of $\theta $.

Consider the identity map $i: A\obp B \ra A\oh B$, and let $z \in (A \obp \C 1)'$. It is easily
seen that $i(z)x = x i(z)$ for all $x\in A \ot \C 1$; so that $
i(z) \in (A \oh \C 1)'$, and by \cite[Corollary 4.7]{smith}, we have $ (A\oh \C 1)' =
\Z(A) \oh B$.  Now, let $i^\prime: \Z(A) \obp B \ra \Z(A) \oh
B$ be the canonical injective homomorphism (like the map $i$). Then, the
following diagram
$$
\xymatrix{
\Z(A)\obp B \ar[rr]^\theta \ar[rd]_{i^\prime} && (A\obp \C 1)' \ar[ld]^{i}\\
 & \Z(A)\oh B} 
$$ commutes.

Note that, the
map $i^\prime$ is surjective as well. To see this, consider an element
$z^\prime \in \Z(A) \oh B$ and fix a sequence $\{z_n\} \seq
\Z(A)\ot B$ such that $\|z_n -z^\prime\|_h \ra 0 $. Since $\Z(A)$ is subhomogenous (being commutative), by \cite[Theorem 6.1]{kumar-sin} we have
$$\|x\|_\gamma \leq K \|x\|_h \ \text{for all}\ x\in \Z(A)\ot B,$$
for some $K>0$. Thus, the sequence $\{z_n\}$
is Cauchy with respect to $\|\cdot\|_{\gamma}$ and converges to some
$z^{\prime\prime}$ in $\Z(A) \obp B$. This shows that $\{z_n =
i'(z_n)\}$ converges to $z'$ as well as to $z''$ in $\Z(A) \oh
B$. So, $i^\prime(z^{\prime\prime}) = z^\prime$ and
$i^\prime$ is surjective.

Thus, for above $i(z)$ in  $(A \oh \C 1)'$,
there exists some $w \in \Z(A) \obp B$ such that $i(z) =
i^\prime(w) = i( \theta(w))$. Since $i$ is injective, $z= \theta(w)$,
so that $\theta $ is surjective and we are done. 
\end{proof}

Recall that for operator spaces $V$ and $W$, and $u\in  M_n(V\otimes W), n\in \N$, the {\it operator space projective tensor norm} is defined as
$$\|u\|_{\wedge}= \inf \{ \|\alpha\| \|v\| \|w\| \|\beta\| : u = \alpha (v\otimes w) \beta  \}, $$
where $\alpha \in \M_{n,pq}, \beta \in \M_{pq,n}, v\in M_p(V)$ and $w\in M_q(W)$, $p,q\in \N$ being arbitrary, and $v\otimes w = (v_{ij} \otimes w_{kl})_{(i,k),(j,l)} \in M_{pq}(V\ot W) $. The {\it operator space projective tensor product} $V\oop W$ is the completion of $V\otimes W$ under $\|\cdot\|_{\wedge}$-norm. For $C^*$-algebras $A$ and $B$, $A\oop B$ is a Banach $*$-algebra. Note that for the operator space projective tensor product, due to the lack of its injectivity when restricted to the tensor product of $C^*$-algebras, the inclusion map discussed in the above result does not extend to an isometry. However, following the same steps with some modifications, we can obtain the isomorphism between the two spaces.

\begin{theorem}
 Let $A$ be a $C^*$-algebras and $B$ be any subhomogenous unital $C^*$-algebra. Then $(A\oop \C 1)'$ is $*$-isomorphic to $\Z(A) \oop B.$
\end{theorem}
 
 \begin{proof}
  Clearly the inclusion function from $ \Z(A) \ot B$ into $ (A \oop \C 1)'$ can be extended to a contractive $*$-homomorphism $\theta$ from $\Z(A) \oop B$ into $(A\oop \C 1)'$. The surjectivity of $\theta$ follows exactly on the same lines using the injectivity of $i$ (\cite[Corollary 1]{rjak}) and  the equivalence between $\oh$ and $\oop$ (\cite[Theorem 7.4]{kumar-sin}) at the appropriate places. For the injectivity of $\theta$, note that $\theta$ is faithful on $\Z(A) \ot B$, so it is also faithful on $\Z(A) \oop B$, see \cite[Theorem 2]{rjak}.
 \end{proof}

\end{document}